\documentclass[11pt, a4paper]{amsart}
\usepackage{a4wide}
\usepackage{amssymb}
\usepackage{amsthm}
\usepackage{amsfonts}
\newtheorem{lemma}{\bf Lemma}[section]

\newtheorem{proposition}[lemma]{\bf Proposition}

\newtheorem{definition}[lemma]{\bf Definition}

\newcommand{\Sub}{{\mathrm{Sub}}}
\newcommand{\G}{{\mathfrak{G}}}

\begin{document}

\title{On incomplete lattice homomorphisms in subspaces of geometries: ``Half'' a problem of Hartmanis from 1959}

\author{Jonathan David Farley}
\address{Department of Mathematics, Morgan State University, 1700 E. Cold Spring Lane, Baltimore, MD 21251, United
States of America}
\email{\tt lattice.theory@gmail.com}
\author{Dominic van der Zypen}
\address{Federal Office of Social Insurance, CH-3003 Bern,
Switzerland}
\email{\tt dominic.zypen@gmail.com}


\begin{abstract} Turing Award winner Juris Hartmanis introduced in 1959 \cite{Ha}
lattices of subspaces of generalized partitions 
(``partitions of type $n$''; ``geometries'' if $n=2$). Hartmanis states it is  ``an unsolved problem whether there are any incomplete lattice homomorphisms in'' lattices of subspaces of geometries.  (He continues, ``[I]f so how can these geometries be characterized.'') We give a positive answer to this question. 

\end{abstract}

\maketitle
\parindent = 0mm
\parskip = 2 mm
\section{Generalized partitions}
A {\em partition of type $n$} for $n\geq 1$
on a set $S$ (consisting of at least $n$ elements) is a set ${\mathfrak P}\subseteq 
{\mathcal P}(S)$ such that
\begin{enumerate}
\item all members of $\mathfrak P$ have at least $n$ elements, and
\item any $n$ elements of $S$ are contained in exactly one member of $\mathfrak P$.
\end{enumerate}
Partitions of type $1$ are the ``traditional'' partitions.

A partition of type $2$ is referred to as a {\em geometry}, and its 
elements are called {\em lines}.

\begin{definition}
If $\G$ is a geometry on a set $S$, a set $T \subseteq S$ is said to be a 
{\em subspace of $S$ with respect to $\G$} if it is ``closed under lines,'' 
that is, for any distinct $x,y \in T$, for the (unique) element $g\in \G$ 
that satisfies $\{x,y\}\subseteq g$ we have $g\subseteq T$. 

We denote the collection of subspaces of $S$ with respect to 
the geometry $\G$ by $\Sub(S,\G)$.
\end{definition}

If $A\subseteq \Sub(S,\G)$ it is easy to see that $\bigcap A \in 
\Sub(S,\G)$, therefore $\Sub(S,\G)$ is a complete lattice
with respect to set inclusion.

\section{Incomplete lattice homomorphisms}
For the terms used in this section, we refer to \cite{DP}. Let $K, L$
be complete lattices. If $f:K\to L$ is order-preserving and $S\subseteq L$
we have $$f(\bigvee_K S) \geq f(s) \textrm{ for all } s\in S,$$ which implies
$$\bigvee_L f(S) \leq f(\bigvee_K S).$$ A lattice homomorphism $f:K\to L$
is said to be {\em join-incomplete} if there is $S\subseteq K$ non-empty
such that $\bigvee_L f(S) < f(\bigvee_K S)$. (Dually, we define 
{\em meet-incompleteness}.) We say $f$ is {\em incomplete} if it is join-incomplete,
meet-incomplete, or both.

The next lemmas deal with incomplete lattice homomorphisms in the
context of infinite complete distributive lattices.
\begin{lemma} \label{lemma1} Let $L$ be an infinite complete and distributive lattice with
bottom element $0$ and top element $1$. Suppose
$P\subseteq L$ is a non-principal prime ideal (i.e., $\bigvee P \notin P$).
Then there is
a join-incomplete 
lattice homomorphism $f: L\to L$ preserving $0$ and $1$.
\end{lemma}
\begin{proof}
Let $f:L\to L$ be $0$ on $P$ and $1$ on $L\setminus P$. Since $P$ is a prime
ideal, $L\setminus P$ is a filter, which implies that $f$ is a lattice homomorphism.
It is (join-)incomplete, because $\bigvee P \notin P$ implies $\bigvee f(P) = 0 
\neq 1 = f(\bigvee P)$.
\end{proof}
Of course, there is a  dual version of Lemma \ref{lemma1} about filters
instead of ideals. Next, we show that there is always either a non-principal
prime ideal or filter in infinite distributive complete lattices.
\begin{lemma}\label{lemma2} If $L$ is infinite, complete, and distributive,
 then it contains either a non-principal
prime ideal or a non-principal prime filter.
\end{lemma}
\begin{proof}
Any infinite distributive lattice contains at least a non-principal ideal
or a non-principal filter. We may assume that $J$ is a non-principal ideal,
so that $j^* = \bigvee J \notin J$. Let $G = \{y\in L: y \geq j^*\}$ be the principal
filter generated by $j^*$. As $J\cap G = \emptyset$ we can
use the Prime Ideal Theorem (see \cite{DP},
Theorem 10.18) and get a prime ideal
$P$ such that $J\subseteq P$ and $P\cap G = \emptyset$, which 
implies $j^* \notin P$.

Next we show that $P$ is not principal: if we had  $p^*:=\bigvee P \in P$ then $J\subseteq P$
would imply $p^* \geq j^* = \bigvee J$ and $j^* \in P$ because $P$ is a down-set.
Therefore $P$ is a non-principal prime
ideal.
\end{proof}
\begin{proposition} \label{prop} Let $L$ be an infinite complete and distributive lattice with
bottom element $0$ and top element $1$. Then there is
an incomplete lattice homomorphism 
$f: L\to L$ respecting $0$ and $1$.
\end{proposition}
\begin{proof}
Combine Lemmas \ref{lemma1} and \ref{lemma2}.
\end{proof}

\section{Construction of an example}
Turing Award winner Juris Hartmanis' problem is on p. 106 of his paper \cite{Ha}:

\begin{quote}
So far we have characterized the complete homomorphisms of the lattices
of subspaces of geometries. It remains an unsolved problem whether there
are any incomplete homomorphisms in these lattices and if so how can
these geometries be characterized.
\end{quote}

In this section we tackle the first part of Hartmanis' problem.

It asks whether 
there is a set $S$ and geometry $\G$ on $S$ and an {\bf incomplete}
lattice homomorphism $$f:\Sub(S, \G)\to \Sub(S, \G).$$

Let $S = \omega$ and set $\G = \big\{\{m,n\}: m, n\in \omega
\land m\neq n\big\}$.

It is easy to see that $\Sub(\omega, \G) = {\mathcal P}(\omega)$.

Since ${\mathcal P}(\omega)$ is distributive, Proposition \ref{prop}
shows that it allows an incomplete lattice endomorphism.

In fact, we can give a more constructive way of providing an incomplete
lattice homomorphism $f:{\mathcal P}(\omega) \to {\mathcal P}(\omega)$. Let $F\subseteq 
{\mathcal P}(\omega)$ denote the set of finite subsets of $\omega$ and
let $M$ denote any maximal ideal containing $F$. Let $f$  send every
member of $M$ to $\emptyset \in {\mathcal P}(\omega)$ and every 
member of ${\mathcal P}(\omega)\setminus M$ to $\omega\in{\mathcal P}
(\omega)$. Then $f$ is an incomplete lattice homomorphism.

What remains open is to have a characterization of the geometries such 
that the complete lattice of subspaces allows incomplete
endomorphisms.




{\footnotesize
\begin{thebibliography}{99}
\bibitem{DP} B. A. Davey Brian and H. A. Priestley, {\bf Introduction to Lattices and 
Order} (second edition), Cambridge University Press, 2002.
\bibitem{Ha} Juris Hartmanis, {\it Lattice Theory of 
Generalized Partitions}, Canadian Journal of  Mathematics
{\bf 11} (1959), 97-106.
\end{thebibliography}
}
\end{document}